\documentclass[12pt,a4paper,reqno,oneside]{amsart}

\usepackage{cite}

\newcommand{\dd}{\mathrm{d}}

\newcommand{\spec}{\mathop{\mathrm{spec}}}
\newcommand{\spece}{\mathop{\mathrm{spec}}\nolimits_\mathrm{ess}}
\newcommand{\specd}{\mathop{\mathrm{spec}}\nolimits_\mathrm{disc}}

\newcommand{\ZZ}{\mathbb{Z}}
\newcommand{\RR}{\mathbb{R}}

\newcommand{\NN}{\mathbb{N}}

\newcommand{\SSS}{\mathbb{S}}

\newcommand{\cO}{\mathcal{O}}
\newcommand{\cD}{\mathcal{D}}

\newcommand{\supp}{\mathop{\mathrm{supp}}}

\newtheorem{theorem}{Theorem}

\newtheorem{cor}[theorem]{Corollary}
\newtheorem{lemma}[theorem]{Lemma}

\theoremstyle{definition}

\newtheorem{rem}[theorem]{Remark}

\begin{document}

\title[]{On the discrete spectrum of Robin Laplacians in~conical domains}

\author{Konstantin Pankrashkin}

\address{Laboratoire de math\'ematiques (UMR 8628 du CNRS), Universit\'e Paris-Sud, B\^atiment 425, 91405 Orsay Cedex, France}

\email{konstantin.pankrashkin@math.u-psud.fr}
\urladdr{http://www.math.u-psud.fr/~pankrash/}

\begin{abstract}
We discuss several geometric conditions guaranteeing the finiteness or the infiniteness
of the discrete spectrum for Robin Laplacians on conical domains.
\end{abstract}

\keywords{Laplacian, Robin boundary condition, eigenvalue, spectrum}

\subjclass{35P15, 35J05, 49R05, 58C40}

\maketitle

\section{Introduction}

Let $\Omega\subset \RR^\nu$ be an open set with a sufficiently regular
boundary. For $\alpha>0$, denote by $Q^\Omega_\alpha$ the self-adjoint operator in $L^2(\Omega)$
acting as the Laplacian $u\mapsto -\Delta u$ in $\Omega$ with the Robin boundary condition
$\partial u/\partial n=\alpha u$ on $\partial\Omega$, where $\alpha>0$ is a fixed constant
and $\partial/\partial n$ means the derivative in the direction of the outer unit normal $n$.
The above definition should be understood in a suitable weak sense:
$Q^\Omega_\alpha$ is the self-adjoint operator in $L^2(\Omega)$
generated by the quadratic form
\[
q^\Omega_\alpha(u,u)=\int_{\Omega} |\nabla u|^2\dd x -\alpha \int_{\partial\Omega} u^2\dd \sigma,
\quad \cD(q^\Omega_\alpha)=H^1(\Omega),
\]
where $\sigma$ stands for the $(\nu-1)$-dimensional Hausdorff measure on $\partial\Omega$.
The study of the spectral properties of $Q^\Omega_\alpha$ arises in numerous applications.
For example, the paper~\cite{lacey} shows a link between the eigenvalue problem and the
long-time dynamics related to some reaction-diffusion process. The paper \cite{bol} discusses the stochastic meaning of the eigenvalues.
Various properties of $Q^\Omega_\alpha$ appear to be of importance
for problems of surface superconductivity, see~\cite{gs1,ow}.
In the present contribution, we are going to discuss some spectral properties of $Q^\Omega_\alpha$ for a special class
of non-compact domains~$\Omega$.

By a \emph{cone} we mean a connected Lipschitz domain $\Lambda\subset\RR^\nu$, $\nu\ge 2$, for which
there exists a point $V\in\RR^\nu$ (vertex of the cone) such that for any $x\in\RR^\nu$
the condition $V+x\in\Lambda$ implies $V+tx\in\Lambda$ for all $t> 0$.
A cone is called \emph{smooth} if its boundary is non-empty and is $C^2$ with the possible
exception of a vertex. Clearly, for $\nu\ge 3$ a cone $\Lambda$ with a vertex $V$ is smooth iff its \emph{cross section}
\[
\Sigma_\Lambda:=\{y\in \RR^\nu, |y|=1:\, V+y\in \Lambda \}\subset \SSS^{\nu-1}
\]
is a domain of $\SSS^{\nu-1}$ with a non-empty $C^2$ boundary, while in $\RR^2$
any cone is smooth. A Lipschitz domain $\Omega$ is called  a \emph{conical} one
if there exists a cone $\Lambda$ such that the $\Omega$ coincides with $\Lambda$ outside a ball, and
the associated cone $\Lambda\equiv\Lambda(\Omega)$ is then uniquely defined.
A conical domain is said to be \emph{smooth at infinity} if the associated
cone is smooth.

The Robin Laplacians on cones play a special role in the study of the bottom of the spectrum $E_1^\Omega(\alpha)$
of $Q^\Omega_\alpha$
as $\alpha$ becomes large. As shown by Levitin and Parnovski \cite{lp}, for piecewise smooth $\Omega$
one has
\begin{equation}
       \label{eq-asymp}
E_1^\Omega(\alpha)=-C_\Omega \alpha^2 +o(\alpha^2) \text{ as $\alpha$ tends to $+\infty$}
\end{equation}
with
$C_\Omega=-\inf_{x\in\partial\Omega} E_1^{\Lambda_x}(1)$, where $\Lambda_x$ is the tangent cone
at $\partial\Omega$ at the point $x$. The paper \cite{lp} also presented several cases
for which the quantity $E_1^{\Lambda_x}(1)$ can be calculated explicitly using some constructions of the spherical geometry.
We remark that there are various refinements of the asymptotics \eqref{eq-asymp}, see e.g. \cite{DaKe10,em,ExMinPar14,hp,hk,LouZhu04,Pank13,kp14,pp,pp15}.

On the other hand, conical geometries attracted an attention in the study of Laplace-type operators
due to the fact that they may produce infinitely many discrete eigenvalues.
It seems that such an effect was found for the first time by Exner and Tater \cite{ET}
who showed that the Dirichlet Laplacian in a rotationally symmetric conical layer
in three dimensions has an infinite discrete spectrum. Recently, Dauge, Ourmi\`eres-Bonafos and Raymond \cite{ROBD} provided a more detailed study by calculating the accumulation rate of the eigenvalues. Behrndt, Exner and Lotoreichik \cite{BEL}
showed the infiniteness of the discrete spectrum for the Schr\"odinger operators with $\delta$-interaction
on rotationally symmetric conical surfaces or their compactly supported perturbations.
We remark that all these works used the invariance of the domain with respect to the rotations.

The principal aim of the present note is to discuss the cardinality of the discrete spectrum
for Robin Laplacians in conical domains smooth at infinity. In particular, we will show that
the infiniteness of the discrete spectrum can be guaranteed by a rather simple sufficient condition:
the strict positivity of the mean curvature of the boundary on a half-line originating from a vertex, see Theorem~\ref{thhh}.
Furthermore, in three dimensions this condition is (in a sense) necessary as well, see Corollary~\ref{corlast3}.
Our constructions are based on a reduction of the analysis of conical domains to
the case of strongly coupled Robin Laplacians carried out by Pankrashkin and Popoff~\cite{pp15}.

Recall that the principal curvatures of a $(\nu-1)$-dimensional smooth submanifold $S$ in $\RR^\nu$
at a point $s$ are defined as the eigenvalues of the shape operator $\dd n(s)$, where $s\mapsto n(s)$
is a smooth unit normal and $\dd$ stands for the differential. The mean curvature $H(s)$ is then the arithmetic mean
of the principal curvatures, i.e. $H(s):=\frac{1}{\nu-1}\mathop{\mathrm{tr}} \dd n(s)$.
The definition depends on the orientation (the change of the orientation implies the sign change
for the curvatures), but in all the cases below we deal with surfaces which are boundaries of open sets,
and the mean curvature will always be calculated with respect
to the \emph{outer} unit normal. In particular, the principal curvatures and the mean curvature of
the boundary of a convex domain are non-negative.

Let us also recall the min-max principle for the eigenvalues. Let $Q$ be a self-adjoint operator
semibounded from below and $q$ be its quadratic form. Let $N\in \NN$ and
\[
E_N(Q):=\inf_{\substack{L\subset \cD(q)\\ \dim L= N}} \sup_{u\in L,\, u\ne 0} \dfrac{q(u,u)}{\|u\|^2},
\]
then
\begin{itemize}
\item  either $E_N(Q)<\inf\spece Q$ and $E_N(Q)$ is the $N$th eigenvalue of $Q$ when numbered in the non-decreasing order
and counted with multiplicities,
\item or $E_k(Q)=\inf\spece Q$ for all $k\ge N$.
\end{itemize}

\section{Essential spectrum}

In order to be able to use the min-max principle we give
first a quite standard result on the essential spectrum.

\begin{theorem}\label{thm1a}
Let $\nu \ge 2$ and $\Omega\subset\RR^\nu$ be a conical domain smooth at infinity, then 
$\spece Q^\Omega_\alpha=[-\alpha^2,+\infty)$ for any $\alpha>0$.
\end{theorem}

\begin{proof}
An easy adaptation of Persson's theorem \cite[Section 14.4]{hs} gives
\[
      \label{eq-pears}
\inf \spece Q^\Omega_\alpha=
\sup_{K} \inf \Big\{
q^\Omega_\alpha(u,u): \, u\in C^\infty_c(\RR^\nu\setminus K), \,
\|u\|_{L^2(\Omega)}=1
\Big\},
\]
where the supremum is over all compact subsets $K\subset\RR^\nu$.
Consider the cone $\Lambda:=\Lambda(\Omega)$ and, furthermore, let $\Theta$
be a \emph{$C^2$ smooth} domain coinciding with $\Omega$ (and, hence, with $\Lambda$) outside a ball.
The above characterization of the bottom of the essential spectrum implies
\begin{equation}
   \label{eq-ess1}
\inf \spece Q^\Omega_\alpha
=
\inf \spece Q^\Theta_\alpha
=
\inf \spece Q^\Lambda_\alpha
\end{equation}
for any $\alpha$. We may use some results of~\cite{pp15} to study $Q^\Theta_\alpha$.
In particular, see \cite[Section~7]{pp15}, for $\alpha\to +\infty$ one has
\[
\inf \spece Q^\Theta_\alpha\ge -\alpha^2- (\nu-1)H_\infty \alpha +o(\alpha),
\quad
H_\infty:=\limsup_{s\to\infty,\,s\in\partial\Omega} H(s),
\]
where $H$ is the mean curvature. In our case one has clearly $H_\infty=0$, hence, 
$\inf \spece Q^\Lambda_\alpha=\inf \spece Q^\Theta_\alpha\ge -\alpha^2+o(\alpha)$.
On the other hand, using the invariance of $\Lambda$ with respect to the dilations we obtain 
\[
\inf\spece Q^\Lambda_\alpha= \alpha^2 \inf\spece Q^\Lambda_1,
\]
which gives $\inf\spece Q^\Lambda_1\ge -1$ and then $\inf \spece Q^\Lambda_\alpha\ge -\alpha^2$ for any $\alpha>0$. Finally, using \eqref{eq-ess1},
\begin{equation}
   \label{eq-q1}
\spece Q^\Omega_\alpha \subset [-\alpha^2,+\infty) \text{ for any } \alpha>0.
\end{equation}

Now we are going to show the inclusion
\begin{equation}
   \label{eq-q2}
	[-\alpha^2,+\infty)\subset \spec Q^\Omega_\alpha.
\end{equation}
This can be easily done by constructing approximate eigenfunctions.
Namely, let $\phi \in C^\infty(\RR)$ with $\phi(x)=0$ for $x\le 0$
and $\phi(x)=1$ for $x\ge 1$. For $N\in\NN$, set
\[
\psi_N(x):=\phi(x-N) \phi(2N-x), \quad
\varphi_N(x):= \phi(\sqrt N -x).
\]
Let $V$ be a vertex of the cone $\Lambda$. Denote by $(r,\theta)\in \RR_+\times \SSS^{\nu-1}$
the spherical coordinates centered at $V$. Furthermore, for $\theta\in \SSS^{\nu-1}$, let $d(\theta)$
denote the geodesic distance in $\SSS^{\nu-1}$ between $\theta$ and the boundary $\partial \Sigma$
of the cross section $\Sigma=\Sigma_\Lambda$. Take an arbitrary $k> 0$ and consider
the functions $u_N$ given in the spherical coordinates by
\[
u_N(r,\theta)=\psi_N(r)\sin(kr) e^{-\alpha r d(\theta)} \varphi_N\big(rd(\theta)\big).
\]
One easily checks that for large $N$  one has $u_N\in \cD(Q^\Omega_\alpha)$ with
$\|u_N\|_{L^2(\Omega)}\ge c N^{\frac{\nu-1}{2}}$,
where $c>0$ is independent of $N$.
Recall that the Laplacian in the spherical coordinates takes the form
\[
-\Delta \simeq -\dfrac{\partial^2}{\partial r} -\dfrac{\nu-1}{r} \dfrac{\partial}{\partial r} -\dfrac{1}{r^2} \Delta_{\SSS^{\nu-1}},
\]
where $\Delta_{\SSS^{\nu-1}}$ is the Laplace-Beltrami operator on $\SSS^{\nu-1}$.
One easily checks that
for the functions $f$ of the form $f(\theta):=b\big(d(\theta)\big)$ we have,
if $d(\theta)<\varepsilon$ and $\varepsilon$ is small,
\[
\Delta_{\SSS^{\nu-1}} f (\theta)= \big(1+\cO(\varepsilon)\big) b''\big(d(\theta)\big)
+\cO(\varepsilon) b'\big(d(\theta)\big).
\]
In particular, for large $N$ we have
\begin{multline*}
\Delta_{\SSS^{\nu-1}} u_N(r,\theta) = \psi_N(r)\sin(kr) r^2 e^{-\alpha r d(\theta)} \\
\times\bigg\{
\Big(1+ \cO(N^{-1/2})\Big) \cdot \Big[\alpha^2  \varphi_N\big(rd(\theta)\big)
-2\alpha \varphi'_N\big(rd(\theta)\big)+\varphi''_N\big(rd(\theta)\big) \Big]\\
+\cO(N^{-3/2})\Big[ -\alpha \varphi_N\big(rd(\theta)\big) + \varphi'_N\big(rd(\theta)\big)\Big]
\bigg\},
\end{multline*}
and elementary estimates give, for large $N$,
\[
\big\|Q^\Omega_\alpha  u_N -(k^2-\alpha^2) u_N \big\|_{L^2(\Omega)}
=\cO(N^{\frac{\nu-2}{2}}).
\]
Therefore,
\[
\dfrac{\big\|Q^\Omega_\alpha  u_N -(k^2-\alpha^2) u_N \big\|_{L^2(\Omega)}}{\|u_N\|_{L^2(\Omega)}}\xrightarrow{N\to+\infty}0,
\]
which gives $k^2-\alpha^2\in \spec Q^\Omega_\alpha$. As $k>0$ is arbitrary and $\spec Q^\Omega_\alpha$
is a closed set, Eq.~\eqref{eq-q2} follows.
As the set $[-\alpha^2,+\infty)$ has no isolated points, Eq. \eqref{eq-q2} implies
the inclusion $[-\alpha^2,+\infty) \subset\spece Q^\Omega$, and the combination with \eqref{eq-q1}
gives the required equality.
\end{proof}

\section{Conical domains with a finite discrete spectrum}

We are going to describe first a class of conical domains whose discrete spectrum
is (at most) finite. We will start with the following assertion which is
essentially an adaptation of the case considered by Exner and Minakov
\cite[Theorem 5.3]{em} in two dimensions.

\begin{theorem}\label{thm2a}
Let $\nu \ge 2$ and $\Omega\subset\RR^\nu$ be an open set with a piecewise $C^2$
boundary such that the complement $\RR^\nu\setminus\Omega$ is convex, then
\[
\spec Q^\Omega_\alpha\subset[-\alpha^2,+\infty) \text{ for any } \alpha>0.
\]
\end{theorem}

\begin{proof}
Let $n$ be the outer unit normal, which is defined at least on the smooth
part $S_0$ of the boundary of $\Omega$. Consider the set $\Theta:=\Phi(S_0\times\RR_+)$ with
\[
\Phi(s,t)=s-tn(s), \quad s\in S_0, \quad t\in \RR_+.
\]
Due to the convexity of $\RR^\nu\setminus \Omega$ we have the inclusion $\Theta\subset \Omega$,
and the map $\Phi$ is a bijection between $S_0\times\RR_+$ and $\Theta$.
Furthermore, passing to the coordinates $(s,t)\in S_0\times\RR_+$ one obtains, for any $u\in C^\infty_c(\RR^\nu)$,
\begin{multline*}
\int_\Omega |\nabla u|^2\dd x-\alpha \int_{\partial\Omega}u^2\dd\sigma
+\alpha^2 \int_\Omega u^2\dd x\\
=\int_\Omega |\nabla u|^2\dd x-\alpha \int_{S_0}u^2\dd\sigma
+\alpha^2 \int_\Omega u^2\dd x\\
 \ge \int_\Theta |\nabla u|^2\dd x-\alpha \int_{S_0}u^2\dd\sigma
+\alpha^2 \int_\Theta u^2\dd x\\
=\int_{S_0} \int_{\RR_+} \big|(\nabla u)\big(\Phi (s,t)\big) \big|^2 J(s,t)  \dd t \,\dd \sigma(s)
-\alpha \int_{S_0}u \big(\Phi(s,0)\big)^2\dd\sigma\\
+\alpha^2 \int_{S_0} \int_{\RR_+} u\big(\Phi (s,t)\big)^2 J(s,t)  \dd t \,\dd \sigma(s)=:A,
\end{multline*}
where $J(s,t)=\big(1- t k_1(s)\big)\cdot\ldots\cdot \big(1- t k_{\nu-1}(s)\big)$
with $k_j$ being the principal curvatures of the boundary $\partial\Omega$.
Due to the convexity of $\RR^\nu\setminus\Omega$
we have $k_j\le 0$ and, subsequently, $J\ge 1$, which gives
\begin{align*}
A\ge &\int_{S_0} \int_{\RR_+} \big|(\nabla u)\big(\Phi (s,t)\big) \big|^2 \dd t \,\dd \sigma(s)
-\alpha \int_{S_0}u \big(\Phi(s,0)\big)^2\dd\sigma(s)\\
&\quad +\alpha^2 \int_{S_0} \int_{\RR_+} u\big(\Phi (s,t)\big)^2 \dd t \,\dd \sigma(s)\\
&\ge
\int_{S_0} \int_{\RR_+} \Big|\big\langle n(s),(\nabla u)\big(\Phi (s,t)\big) \big\rangle\Big|^2 \dd t \,\dd \sigma(s)
-\alpha \int_{S_0}u \big(\Phi(s,0)\big)^2\dd\sigma(s)\\
&\quad +\alpha^2 \int_{S_0} \int_{\RR_+} u\big(\Phi (s,t)\big)^2 \dd t \,\dd \sigma(s)\\
&=
\int_{S_0} \int_{\RR_+} \Big|\dfrac{\partial}{\partial t}\,u\big(\Phi (s,t)\big) \Big|^2 \dd t \dd \sigma(s)
-\alpha \int_{S_0}u \big(\Phi(s,0)\big)^2\dd\sigma\\
&\quad +\alpha^2 \int_{S_0} \int_{\RR_+} u\big(\Phi (s,t)\big)^2 \dd t \dd \sigma(s)\\
&=\int_{S_0} \bigg(
\int_{\RR_+} \Big|\dfrac{\partial}{\partial t}\,u\big(\Phi (s,t)\big) \Big|^2 \dd t
-\alpha u \big(\Phi(s,0)\big)^2 + \alpha^2 \int_{\RR_+} u\big(\Phi (s,t)\big)^2 \dd t
\bigg) \dd\sigma(s).
\end{align*}
As for each $v\in H^1(\RR_+)$ we have
\[
\int_{\RR_+} v'(t)^2\dd t -\alpha v(0)^2 + \alpha^2 \int_{\RR_+} v(t)^2 \dd t \ge 0,
\]
we arrive at $A\ge 0$ and
\[
\int_\Omega |\nabla u|^2\dd x-\alpha \int_{\partial\Omega}u^2\dd\sigma
+\alpha^2 \int_\Omega u^2\dd x\ge 0, \quad u \in C^\infty_c(\RR^\nu).
\]
By density, this extends to all $u\in H^1(\Omega)$ and gives $Q^\Omega_\alpha\ge-\alpha^2$.
\end{proof}

By combining Theorems~\ref{thm1a} and \ref{thm2a} we obtain the following result:
\begin{cor}\label{cor1}
Let $\nu\ge 2$ and $\Omega\subset\RR^\nu$ be a conical domain smooth at infinity
and such that $\RR^\nu\setminus\Omega$ is convex, then
$\spec Q^\Omega_\alpha=[-\alpha^2,+\infty)$ for any $\alpha>0$.
In particular, $Q^\Omega_\alpha$ has no discrete spectrum for any $\alpha>0$.
\end{cor}

A slight modification of the above discussion gives a sufficient condition
for the finiteness of the discrete spectrum in more general situations.

\begin{theorem}\label{thm2c}
Let $\nu\ge 2$ and $\Omega\subset\RR^\nu$ be a conical domain smooth at infinity
and $\Lambda=\Lambda(\Omega)$ be the associated cone.
Assume that the complement $\RR^\nu\setminus \Lambda$ is
convex, then the discrete spectrum of $Q^\Omega_\alpha$
is finite for any $\alpha>0$.
\end{theorem}

\begin{proof} 
For $r>0$, denote $B_r:=\{x\in\RR^\nu: |x|<r\}$.
Assume that the origin is a vertex of $\Lambda$
and choose $R>0$ such that $\Omega$ coincides with $\Lambda$
outside $B_{R-1}$. Denote $\Omega_i:=B_R\mathop{\cap} \Omega$ and
$\Omega_e:=\Omega\setminus \overline{\Omega_i}$ and consider the
quadratic form
\begin{align*}
\widetilde q(u,u)&:=\int_{\Omega_i} |\nabla u|^2 \dd x -\alpha \int_{\partial \Omega_i\cap \partial \Omega} u^2 \dd\sigma\\
&\quad +\int_{\Omega_e} |\nabla u|^2 \dd x -\alpha \int_{\partial \Omega_e\cap \partial \Omega} u^2 \dd\sigma,
\quad u \in H^1(\Omega_i\mathop{\cup}\Omega_e),
\end{align*}
which generates a self-adjoint operator $\widetilde Q$ in $L^2(\Omega)$.
As $\widetilde q$ extends $q^\Omega_\alpha$, due to the min-max principle
it is sufficient to show that $\widetilde Q$ has finitely many eigenvalues
in $(-\infty,-\alpha^2)$.
Remark that  $\widetilde Q=Q_i\oplus Q_e$, where
$Q_i$ and $Q_e$ are the self-adjoint operators acting respectively in $L^2(\Omega_i)$
and $L^2(\Omega_e)$ and generated by the quadratic forms
\begin{align*}
q_i(u,u)&:=\int_{\Omega_i} |\nabla u|^2 \dd x -\alpha \int_{\partial \Omega_i\cap \partial \Omega} u^2 \dd\sigma, \quad u \in H^1(\Omega_i),\\
q_e(u,u)&:=\int_{\Omega_e} |\nabla u|^2 \dd x -\alpha \int_{\partial \Omega_e\cap \partial \Omega} u^2 \dd\sigma,
\quad u \in H^1(\Omega_e)
\end{align*}
respectively.
Denote $S_e:=\partial \Omega_e\cap \partial\Omega\subset\partial\Lambda\mathop{\cap}\partial\Omega$
and consider the map
\[
S_e\times \RR_+ \ni (s,t)\mapsto \Phi(s,t):=s- t n(s),
\]
where $n$ is the outer unit normal.
Due to the convexity of $\RR^\nu\setminus\Lambda$ the map $\Phi$ is injective, and $\Phi(s,t)\in \Lambda$
for any admissible $(s,t)$. On the other hand, as $S_e\subset \partial \Lambda$,
we have $s\mathop{\bot} n(s)$ and, hence, $\big|\Phi(s,t)\big|\ge |s|> R$, i.e.
$\Phi(s,t)\in \RR^\nu\setminus \overline{B_R}$. Therefore, $\Phi(S_e\times\RR_+)\subset \Lambda\setminus \overline{B_R}\equiv  \Omega_e$.
Proceeding as in the proof of Theorem~\ref{thm2a} we obtain $Q_e\ge -\alpha^2$, which means that
the spectrum of $\widetilde Q$ in $(-\infty,-\alpha^2)$
coincides with that of $Q_i$. On the other hand $Q_i$ is semibounded from below
and  has a compact resolvent and, hence, can have only finitely many eigenvalues
in this interval, which gives the sought result.
\end{proof}

\section{Conical domains with an infinite discrete spectrum}

Now we are going to go in the opposite direction and to show that
Robin laplacians on a large class of conical domains have an infinite discrete spectrum.

\begin{lemma}\label{eq-lemma}
Let $\nu\ge 3$ and $\Lambda\in \RR^\nu$ be a smooth cone with a vertex at the origin.
Assume that there exists a point on $\partial\Lambda\setminus\{0\}$ at which the mean curvature
is strictly positive. Then for any $\alpha>0$, $r_0>0$, $N\in \NN$ there exist $N$ functions
$F_j \in H^1(\Lambda)\setminus\{0\}$, $j=1,\dots, N$, satisfying the following conditions:
\begin{itemize}
\item[(a)] $F_j(x)=0$ for $|x|\le r_0$ and $j=1,\dots, N$,
\item[(b)] the functions $F_j$ have mutually disjoint supports,
\item[(c)] for any $j=1,\dots, N$ there holds
\begin{equation}
    \label{eq-minmin}
		q^\Lambda_\alpha(F_j,F_j)\le -\alpha^2 \|F_j\|^2_{L^2(\Lambda)}.
\end{equation}
\end{itemize}
\end{lemma}

\begin{proof}
We remark first that if the mean curvature $H$ is positive at some point $P\in \partial\Lambda\setminus\{0\}$, then the same holds
on the whole half-line $\RR_+ P\subset \partial\Lambda$.

For $R>0$, introduce a unitary transform $U_R:L^2(\Lambda)\to L^2(\Lambda)$
by $(U_R f)(x)=R^{\nu/2} f(Rx)$. It is easy to check the identity
$R^2 Q^\Lambda_\alpha= U_R^* Q^\Lambda_{R\alpha} U_R$.
In particular, for $F\in H^1(\Lambda)$ we have
\begin{equation}
     \label{eq-FF}
\dfrac{q^\Lambda_\alpha(F,F)}{\|F\|^2_{L^2(\Lambda)}}=\dfrac{1}{R^2}\dfrac{q^\Lambda_{R\alpha}(U_R F,U_R F)}{\|U_R F\|^2_{L^2(\Lambda)}}.
\end{equation}
We will look for functions $F_j$ with the above properties using the representation $U_R F_j=f_j$, i.e.
$F_j=U^*_R f_j=U_{1/R} f_j$, where $R$ is a large parameter and
$f_j$ are new functions to be chosen in a suitable way.

For $r>0$, denote $B_r:=\{x\in\RR^\nu: |x|<r\}$. Let $b>2$. 
By smoothening the vertex of $\Lambda$ one can construct a connected $C^2$ smooth domain
$\Theta$ coinciding with $\Lambda$ outside the ball $B_{b-1}$.
Clearly, for any function $f\in H^1(\Lambda)$ vanishing in $\Lambda\mathop{\cap}B_b$
we have
\begin{equation}
   \label{eq-RR}
\dfrac{q^\Lambda_{R\alpha}(f,f)}{\|f\|^2_{L^2(\Lambda)}}=\dfrac{q^\Theta_{R\alpha}(\widetilde f,\widetilde f\,)}{\|\widetilde f\|^2_{L^2(\Theta)}},
\end{equation}
where $\widetilde f = f$ on $\Theta\setminus B_{b-1}$ and $\widetilde f=0$ on the remaining part of $\Theta$.
Now we construct a function $\widetilde f$ with the above property in a special way using some computations of~\cite{pp15}.

Denote $S:=\partial\Theta$ and let $S\ni s\mapsto n(s)$ be the outer unit normal.
Furthermore, let $\delta:=\delta(R)$ be a positive function such that $\delta(R) \to 0$ and $R\delta(R)\to +\infty$ as
$R\to +\infty$. Denote by $E$ and $\psi$ respectively the first eigenvalue and the associated normalized
eigenfunction of the operator $\phi\mapsto -\phi''$ in $L^2(0,\delta)$ with the boundary conditions
$\phi'(0)+R\alpha \phi(0)=0$ and $\phi(\delta)=0$. We recall that one has the estimate $-\alpha^2R^2\le E\le -\alpha^2R^2+\cO(\alpha^2R^2e^{-\delta R \alpha})$
as $R\to+\infty$, see \cite[Lemma~2.1]{pp15}.

Consider the map $S\times(0,\delta)\ni(s,t)\mapsto \Phi(s,t):=s-tn(s)\in \RR^\nu$.
If $R$ is sufficiently large, then $\Phi$ is a diffeomorphism between $S\times(0,\delta)$ and $\Phi\big(S\times(0,\delta)\big)$
and, moreover, $\Phi\big(S\times(0,\delta)\big)\subset \Theta$.
For $v\in H^1(S)$, consider the associated functions $\widetilde f=\widetilde f_v\in H^1(\Theta)$ given by
\begin{equation}
    \label{eq-fv}
\widetilde f (x)\equiv \widetilde f_v(x):=\begin{cases} v(s)\psi(t), & x=\Phi(s,t), \quad (s,t)\in S\times(0,\delta),\\
0, & x\in \Theta \setminus \Phi\big(S\times(0,\delta)\big).
\end{cases}
\end{equation}
By \cite[Lemma 4.1]{pp15}, one can find constants $c'_1,c'_2,c'_3$ and $R_0$ such that
for all $v\in H^1(S)$ and $R\ge R_0$ there holds
\begin{multline*}
\dfrac{q^\Theta_{R\alpha}(\widetilde f,\widetilde f)}{\|\widetilde f\|^2_{L^2(\Theta)}}-E\\
\le (1+c'_1 \delta) \dfrac{(1+c'_2\delta) \displaystyle \int_S g^{\rho \mu} \partial_\rho v \,\partial_\mu v \,\dd \sigma -R\alpha (\nu-1)\int_{S} H v^2 \,\dd \sigma}{\|v\|^2_{L^2(S,\dd\sigma)}}\\
+c'_3(1+R\alpha e^{-R\alpha \delta}),
\end{multline*}
where $(g^{\rho\mu})$ is the contravariant metric tensor on $S$ induced by the embedding into $\RR^\nu$
and Einstein convention for indices is used.
In order to simplify the subsequent computations we take $\delta:=R^{-1/2}$, then there are positive constants $b_1$, $b_2$, $b_3$ and $R_1$ such that
for all $v\in H^1(S)$ and $R\ge R_1$ there holds
\begin{equation}
\dfrac{q^\Theta_{R\alpha}(\widetilde f,\widetilde f)}{\|\widetilde f\|^2_{L^2(\Theta)}}
\le -\alpha^2 R^2 + \dfrac{b_1\displaystyle \int_S g^{\rho \mu} \partial_\rho v \,\partial_\mu v \,\dd \sigma - b_2\alpha R \int_{S} H v^2 \dd \sigma}{\|v\|^2_{L^2(S,\dd\sigma)}}
+b_3.
\end{equation}
Choose a point $M\in \partial\Sigma_\Lambda$ with $H(M)=\max_{s\in \partial\Sigma_\Lambda} H(s)$.
By assumption, $H(M)>0$. By applying a suitable rotation we may assume that  $M$ has the coordinates $(1,0,\dots, 0)$
and that $n(M)=(0,\dots, 0, -1)$. It follows that the hyperplane $x_\nu=0$ is tangent to $\Lambda$ at $M$, and
there exists a $C^2$ function $h$ with $h(0,\dots,0)=0$ and $\nabla h(0,\dots,0)=0$ and some $\varepsilon>0$ and $C>0$
such that
\begin{multline*}
\Lambda \cap \{x_1=1\}\cap \bigcap_{k=2}^{\nu-1} \big\{|x_k|< \varepsilon\big\} \cap \big\{|x_\nu|< C\varepsilon\big\}\\
=\big\{x_\nu > h(x_2,\dots,x_{\nu-1}) \big\} \cap \{x_1=1\}\cap \bigcap_{k=2}^{\nu-1} \big\{|x_k|< \varepsilon\big\} \cap \big\{|x_\nu|< C\varepsilon\big\}.
\end{multline*}
It follows that
\begin{multline*}
\partial \Lambda \cap \{x_1>1\}\cap \bigcap_{k=2}^{\nu-1} \big\{|x_k|< \varepsilon x_1\big\} \cap \big\{|x_\nu|< C\varepsilon x_1\big\}\\
=\bigg\{(x_1,\dots,x_\nu): x_1>1, \, |x_k|< \varepsilon x_1, \,k=2,\dots,\nu-1,\\
\qquad x_\nu= x_1 h\Big(\dfrac{x_2}{x_1},\dots,\dfrac{x_{\nu-1}}{x_1}\Big)\bigg\}.
\end{multline*}
As $S$ coincides with $\partial \Lambda$ for $|x|>b$, one can use $(x_1,\dots,x_{\nu-1})$ as local coordinates
on a part $S_{b,\varepsilon}$ of $S$ defined as $S_{b,\varepsilon}:=X(\Pi_{b,\varepsilon})$, where
\begin{gather*}
\Pi_{b,\varepsilon}:= \big\{ (x_1,\dots,x_{\nu-1}): x_1\ge b, \quad |x_k|< \varepsilon x_1, \quad k=2,\dots,\nu-1 \big\},\\
X(x_1,\dots,x_{\nu-1}):=
\begin{pmatrix}
x_1\\
\dots\\
x_{\nu-1}\\
x_1 h\Big(\dfrac{x_2}{x_1},\dots,\dfrac{x_{\nu-1}}{x_1}\Big)
\end{pmatrix}.
\end{gather*}
In particular, for the associated metric tensor $g_{jk}$ we have
\[
g_{jk}(x_1,\dots,x_{\nu-1})=\dfrac{\partial X}{\partial x_j} \cdot \dfrac{\partial X}{\partial x_k}=\delta_{jk}+a_{jk},
\]
where $\delta_{jk}$ is Kronecker delta,
\begin{align*}
a_{11}&= \Big[ h\Big(\dfrac{x_2}{x_1},\dots,\dfrac{x_{\nu-1}}{x_1}\Big) - \sum_{\ell=2}^{\nu-1} \dfrac{x_{\ell}}{x_1} h'_{\ell-1}\Big(\dfrac{x_2}{x_1},\dots,\dfrac{x_{\nu-1}}{x_1}\Big) \Big]^2,\\
a_{1j}=a_{j1}&= \Big[ h\Big(\dfrac{x_2}{x_1},\dots,\dfrac{x_{\nu-1}}{x_1}\Big) - \sum_{\ell=2}^{\nu-1} \dfrac{x_{\ell}}{x_1} h'_{\ell-1}\Big(\dfrac{x_2}{x_1},\dots,\dfrac{x_{\nu-1}}{x_1}\Big) \Big]\\
& \quad \times h'_{j-1}\Big(\dfrac{x_2}{x_1},\dots,\dfrac{x_{\nu-1}}{x_1}\Big), \quad j=2,\dots,\nu-1,\\
a_{jk}&= h'_{j-1}\Big(\dfrac{x_2}{x_1},\dots,\dfrac{x_{\nu-1}}{x_1}\Big) h'_{k-1}\Big(\dfrac{x_2}{x_1},\dots,\dfrac{x_{\nu-1}}{x_1}\Big), \quad j,k=2,\dots,\nu-1,
\end{align*}
and $h'_\ell$ stands for the partial derivative of $h$ with respect to the $\ell$th variable.
In particular, one may assume that the parameter $\varepsilon$ in the above constructions
is sufficiently small to have $\big\|(a_{jk})\big\|\le \frac{1}{2}$, which then implies, for   $x\in \Pi_{b,\varepsilon}$,
\begin{gather}
  \label{eq-est2}
\dfrac{1}{2}\le (g_{jk})\le \dfrac{3}{2},\quad (g^{jk}):=(g_{jk})^{-1} \le 2,\\
  \label{eq-est3}
\Big(\dfrac{1}{2}\Big)^{(\nu-1)/2}\le\sqrt{\det g} \le \Big(\dfrac{3}{2}\Big)^{(\nu-1)/2}.
\end{gather}
In addition, let us pick $A\in \big(0, H(M)\big)$, then due to the continuity of $H$
we may assume that $\varepsilon$ is sufficiently small to have
\[
H\big( X(1,x_2, \dots, x_{\nu-1})\big) \ge A \text{ for } |x_k|\le \varepsilon, \quad k=2,\dots,\nu-1,
\]
and, hence,
\[
H\big( X(x_1,x_2, \dots, x_{\nu-1})\big) \ge \dfrac{A}{x_1}
\text{ for $x\in \Pi_{b,\varepsilon}$}.
\]
For $\varphi\in C^\infty_c(\Pi_{b,\varepsilon})$ define $v\equiv v_\varphi\in H^1(S)$ through
\[
v\big(X(x_1,x_2, \dots, x_{\nu-1})\big):=\varphi(x_1,\dots, x_{\nu-1})
\]
and extend it by zero to the whole of $S$. Using the representation
$\dd \sigma=\sqrt{\det g} \,\dd x_1\dots\dd x_{\nu-1}$
and the estimates \eqref{eq-est2} and \eqref{eq-est3}  we obtain, as $R\ge R_1$,
\begin{multline*}
\dfrac{q^\Theta_{R\alpha}(\widetilde f,\widetilde f)}{\|\widetilde  f\|^2_{L^2(\Theta)}}
\le -\alpha^2 R^2\\
+ \dfrac{c_1\displaystyle\int_{\Pi_{b,\varepsilon}} |\nabla \varphi|^2 \dd x_1\dots\dd x_{\nu-1}
- c_2 \alpha R\displaystyle\int_{\Pi_{b,\varepsilon}} \dfrac{\varphi^2}{x_1} \, \dd x_1\dots\dd x_{\nu-1}}{\displaystyle\int_{\Pi_{b,\varepsilon}} \varphi^2 \dd x_1\dots\dd x_{\nu-1}}
 +c_3
\end{multline*}
with
$c_1:= 2\cdot 3^{(\nu-1)/2}b_1>0$, $c_2:= A\cdot 3^{(1-\nu)/2}b_2>0$ and $c_3:=b_3>0$.
Now let us pick $\phi_j \in C^\infty_c(\RR^{\nu-1})$ with $\supp\, \phi_j\subset \Pi_{b,\varepsilon}$, $j=1,\dots,N$,
not identically zero and with mutually disjoint supports
and set
\begin{gather*}
\varphi_{j,R}(x):=\phi_j\Big(\dfrac{x}{\sqrt R}\Big), \\
m:=\sup\Big\{x_1: \text{ there exists }(x_1,\dots x_\nu)\in \bigcup_{j=1}^N \supp\, \phi_j\Big\}.
\end{gather*}
For $R\ge 1$ we still have $\varphi_{j,R}\in C^\infty_c(\Pi_{b,\varepsilon})$ and these new functions
still have mutually disjoint supports. Furthermore, set $v_{j,R}:=v_{\varphi_{j,R}}$
and $\widetilde f_{j,R}:=\widetilde f_{v_{j,R}}$, $j=1,\dots,N$, see \eqref{eq-fv}, then, by construction, the functions $\widetilde f_{j,R}$
are not identically zero and have mutually disjoint supports, and
\begin{multline*}
\dfrac{q^\Theta_{R\alpha}(\widetilde f_{j,R},\widetilde f_{j,R})}{\|\widetilde f_{j,R}\|^2_{L^2(\Theta)}}
\le -\alpha^2 R^2\\
+ \dfrac{\dfrac{c_1}{R }\displaystyle\int_{\RR^{\nu-1}} |\nabla \phi_j|^2 \dd x_1\dots\dd x_{\nu-1}
- \dfrac{c_2 \alpha \sqrt{R}}{m}\displaystyle\int_{\RR^{\nu-1}} \phi_j^2 \, \dd x_1\dots\dd x_{\nu-1}}{\displaystyle\int_{\RR^{\nu-1}} \phi_j^2 \dd x_1\dots\dd x_{\nu-1}}
+c_3.
\end{multline*}
In particular, for sufficiently large $R$ we have the strict inequalities
\[
\dfrac{q^\Theta_{R\alpha}(\widetilde f_{j,R},\widetilde f_{j,R})}{\|\widetilde  f_{j,R}\|^2_{L^2(\Theta)}} <-\alpha^2 R^2.
\]
Remark that by construction for sufficiently large $R$ one also has $\widetilde f_{j,R}(x)=0$ for $|x|\le b+1$.
In particular, one can define $f_{j,R}\in H^1(\Lambda)$ by $f_{j,R}=\widetilde f_{j,R}$ on $\Lambda\cap\big\{|x|\ge b\big\}$
and $f_{j,R}=0$ on $\Lambda\cap\big\{|x|\le b\big\}$, then by \eqref{eq-RR} we have, as $R$ is large,
\[
\dfrac{q^\Lambda_{R\alpha}(f_{j,R},f_{j,R})}{\|f_{j,R}\|^2_{L^2(\Lambda)}}< -\alpha^2 R^2.
\]
Finally, define $F_j(\cdot )\equiv F_{j,R}(\cdot):=R^{\nu/2}f_{j,R}(\cdot/R)$, then \eqref{eq-minmin} holds
due to \eqref{eq-FF}.
By construction, $F_j$ are not identically zero, have mutually disjoint supports
and vanish in any prescribed  ball if $R$ is large. Hence, the family $(F_j)$ satisfies all the requested conditions.
\end{proof}

\begin{theorem}\label{thhh}
Let $\nu\ge 3$ and $\Omega\subset \RR^\nu$ be a conical domain smooth at infinity.
Assume that outside any ball there is a point at the boundary of $\Omega$
at which the mean curvature is strictly positive, then for any $\alpha>0$
the discrete spectrum of $Q^\Omega_\alpha$ is infinite.
\end{theorem}

\begin{proof}
Consider the cone $\Lambda=\Lambda(\Omega)$. Without loss of generality
one may assume that the origin is a vertex of $\Lambda$. Let $r_0>0$
be such that $\Lambda$ and $\Omega$ coincide in $\big\{x\in \RR^\nu: |x|> r_0-1\big\}$.
Take any $N\in\NN$ and construct functions $F_1,\dots,F_N\in H^1(\Lambda)$ as in Lemma \ref{eq-lemma}.
The subspace $L$ spanned by the functions $F_j$ is then $N$-dimensional.
For $F\in L$, define $\widetilde F\in H^1(\Omega)$ by $\widetilde F=F$ on $\Lambda\cap\Omega$
and $\widetilde F=0$ on the remaining part of $\Omega$. 
We have then $q^\Omega_\alpha(\widetilde F,\widetilde F)=q^\Lambda_\alpha(F,F)$
and $\|\widetilde F\|_{L^2(\Omega)}=\|F\|_{L^2(\Lambda)}$ for any $F\in L$, and the subspace
$\widetilde L:=\{\widetilde F: F\in L\}\subset H^1(\Omega)$ is $N$-dimensional.
Due to the choice of $F_j$ we have
\begin{multline*}
E_N(Q^\Omega_\alpha)\le \sup_{G\in \widetilde L,\, G\not\equiv 0}
\dfrac{q^\Omega_\alpha(G,G)}{\|G\|^2_{L^2(\Omega)}}
=\sup_{F\in L,\, F\not\equiv 0}
\dfrac{q^\Omega_\alpha(\widetilde F, \widetilde F)}{\|\widetilde F\|^2_{L^2(\Omega)}}\\
=\sup_{F\in L,\, F\not\equiv 0}
\dfrac{q^\Lambda_\alpha(F, F)}{\|F\|^2_{L^2(\Lambda)}}
<-\alpha^2=\inf\spece Q^\Omega_\alpha.
\end{multline*}
Hence, $Q^\Omega_\alpha$ has at least $N$ eigenvalues in $(-\infty,-\alpha^2)$
by the min-max principle. As $N$ is arbitrary, the result follows.
\end{proof}

\section{Conical domains in three dimensions}

It appears that the main assumption of Theorem~\ref{thhh} (the positivity of the mean curvature on an unbounded set)
is quite simple to check for three-dimensional cones, as the following assertion shows.

\begin{lemma}\label{thmlast}
Let $\Lambda$ be a cone in $\RR^3$ with a simply connected smooth cross section $\Sigma:=\Sigma_\Lambda\subset \SSS^2$.
Assume that the set $\RR^\nu\setminus\Lambda$ is \emph{not} convex, then there is half-line
on the boundary of $\partial\Lambda$ on which the mean curvature is strictly positive.
\end{lemma}

\begin{proof}
We may assume that the origin is a vertex of $\Lambda$. By assumption, the boundary $\partial \Sigma$ is a simple smooth curve on $\SSS^2$. Let $\ell$ be its length
and $\gamma:\RR/\ell\ZZ\to\SSS^2\subset\RR^3$ be its arc-length parametrization chosen such that
the vector $\gamma\times \gamma'$ points to the outside of $\Lambda$. It is elementary to check that
the map
\[
(\RR/\ell\ZZ)\times \RR_+ \ni (s,t)\mapsto X(s,t):=t\gamma(s) 
\]
gives a parametrization of the boundary of $\Lambda$ and that the mean curvature
$H(s,t)$ at the point $X(s,t)$ is given by $H(s,t)=\kappa(s) / (2 t)$, where
$\kappa(s)=\det\big(\gamma(s),\gamma'(s),\gamma''(s)\big)$ is the geodesic curvature
of $\partial\Sigma$ at the point $\gamma(s)$. If $\kappa\le 0$ at all points, then the complement $\SSS^2\setminus \Sigma$
is geodesically convex, see e.g. \cite[Proposition 2.1]{convex} and then the complement $\RR^3\setminus\Lambda$
in its turn should be convex too, which contradicts the assumption. Therefore, there exists $s_0$
such that $\kappa(s_0)>0$, and then $H(s_0,t)>0$ for all $t\in \RR_+$.
\end{proof}

The combination of Theorem~\ref{thhh} with Lemma~\ref{thmlast} gives the following result:
\begin{cor}\label{corlast}
Let $\Omega\subset\RR^3$ be a conical domain smooth at infinity.
Assume that the associated cone $\Lambda(\Omega)$ has a simply connected cross section
and  that the complement $\RR^3\setminus \Lambda(\Omega)$ is not a convex set,
then $Q^\Omega_\alpha$ has an infinite discrete spectrum for any $\alpha>0$.
\end{cor}

We point out the following fact concerning the three-dimensional smooth cones:

\begin{cor}\label{corlast2}
Let $\Lambda\subset\RR^3$ be a cone with a simply connected smooth cross section and $\alpha>0$, then
the discrete spectrum of $Q^\Lambda_\alpha$ is either infinite or empty.
\end{cor}

\begin{proof} If $\RR^3\setminus\Lambda$ is convex, then the discrete spectrum is empty
by Corollary~\ref{cor1}, otherwise it is infinite by Corollary~\ref{corlast}.
\end{proof}

Finally, the combination of Theorem~\ref{thm2c} with Corollary~\ref{corlast} gives the following result:
\begin{cor}\label{corlast3}
Let $\alpha>0$ and $\Omega\subset\RR^3$ be a conical domain such that the associated cone $\Lambda(\Omega)$
has a simply connected smooth cross section, then
the discrete spectrum of $Q^\Omega_\alpha$ is finite if and only if the complement
$\RR^\nu\setminus\Lambda(\Omega)$ is a convex set.
\end{cor}

\begin{rem}
We remark that there is no analog of Corollaries~\ref{corlast} and~\ref{corlast2} for non-smooth cones.
As an example one can consider $Q^\Omega_\alpha$ for the convex cone $\Omega:=(\RR_+)^\nu\subset \RR^\nu$.
The separation of variables gives, for any $\alpha>0$,
\begin{gather*}
\spece Q^\Omega_\alpha=\big[-(\nu-1) \alpha^2,+\infty\big),\\
\specd Q^\Omega_\alpha=\{-\nu \alpha^2\},\quad
\dim\ker(Q^\Omega_\alpha+\nu \alpha^2)=1.
\end{gather*}
The example also shows that Corollary~\ref{corlast3} cannot be extended to the two-dimensional case.
\end{rem}

\section*{Acknowledgments}

The work was partially supported by GDR CNRS 2279 DYNQUA and
ANR NOSEVOL 2011 BS01019 01.

\end{document}